\documentclass[10pt,twoside]{siamltex}
\usepackage{amsfonts,epsfig}
\usepackage{epstopdf}
\usepackage{subfigure}
\usepackage{amsmath}
\usepackage{color}
\usepackage{enumerate}
\usepackage{tikz}
\usepackage{enumitem}

\newtheorem{example}[theorem]{Example}

\begin{document}



\bibliographystyle{plain}
\title{
A Short Note on Doubly Substochastic Analog of Birkhoff's Theorem}

\author{
Lei Cao\thanks{Department of Mathematics,
Georgian Court University, Lakewood, New Jersey, 08701, USA
(lcao@georgian.edu).}
}

\pagestyle{myheadings}
\maketitle

\begin{abstract}

Let $B$ be an $n\times n$ doubly substochastic matrix. 
We show that $B$ can be written as a convex combination of no more than $\sigma(B)+t$ subpermutation matrices, where $\sigma(B)$ is the number of nonzero elements in $B$ and $t$ is the number of fully indecomposable components of $B^{comp},$ the minimal doubly stochastic completion of $B$ obtained by a specific way.
\end{abstract}

\begin{keywords}
Doubly Substochastic Matrices; Birkhoff-von Neumann Theorem; Subpermutation Matrices;
\end{keywords}
\begin{AMS}
15B51,  15A99.
\end{AMS}

\section{Introduction} \label{intro-sec}
Let $\Omega_n$ be the set all of $n\times n$ doubly stochastic matrices. It is well-known that $\Omega_n$ is a polytope and whose dimension is $(n-1)^2$ and whose extreme points are the $n\times n$ permutation matrices (see \cite{Survey1964} II.1, \cite{BG}). Let $A=[a_{ij}]$ be an $n\times n$ (0,1) matrix. Define the face of $\Omega_n$ corresponding to $A$, denoted by $\mathcal{F}(A)$, as the set of all $n\times n$ doubly stochastic matrices $X=[x_{ij}]$ such that $x_{ij}\leq a_{ij}$ for all $i,j=1,2,\ldots, n.$  In \cite{BRU1977}, Brualdi and Gibson show that the dimension of $\mathcal{F}(A)$ is $\sigma(A)-2n+t$ where $\sigma(A)$ is the number of $1$'s in $A$ and $t$ is the number of fully indecomposable components of $A.$ Therefore by Carath\'{e}odory's theorem, every doubly stochastic matrix in $\mathcal{F}(A)$ can be written as a convex combination of no more than $\sigma(A)-2n+t+1$ permutation matrices. Recently we note that Dufoss\'{e} and U\c{c}ar showed that determining the minimal number of permutation matrices needed is strongly NP-complete (\cite{DU2016}).

Let $\omega_n$ be the set of all $n\times n$ doubly substochastic matrices. It is known that $\omega_n$ is a convex polytope as well and its vertices are the $n\times n$  subpermutation matrices (\cite{MatrixAnalysis}, p. 165). However, there is no upper bound or the number of subpermutation matrices needed for a given doubly substochastic matrix is known. In this paper, we give an upper bound via a combination of the results from \cite{BRU1977} and \cite{CK2016}.

In \cite{CK2016}, Cao, Koyuncu and Parmer defined the sub-defect of doubly substochastic matrices and provided a specific way to obtain a minimal doubly stochastic completion of any doubly substochastic matrix. Such a minimal completion contains the most zeros among all completions and hence minimizes the number of nonzero elements which is the most influential part of the dimension of the face containing its completion.

\section{Properties and Notation}

\begin{definition} Let $B\in \omega_n.$ Then the {\bf sub-defect} of $B,$ denoted by $sd(B),$ is defined to be the smallest integer $k$ such that there exists $A\in\Omega_{n+k}$ containing $B$ as a submatrix.
\end{definition}

Let $M$ be an $n\times n$ matrix. Denote  the sum of all elements of $M$ by $\sigma(M),$  the sum of the $i$th row of $M$ by $r_i(M)$ and the sum of the $i$th column of $M$ by $c_i(M).$  Then $sd(B)$ can be computed easily by taking the difference between the size of $B$ and $\sigma(B)$ as the follows.

\begin{theorem} (\cite{CK2016}, Theorem 2.1 \label{thm2.1}) Let $B\in \omega_n$. Then $$sd(B)=\lceil n-\sigma(B)\rceil,$$  where $\lceil \cdot \rceil$ is the ceiling function.
\end{theorem}

Denote $\omega_{n,k}$ the set of all $n\times n$ doubly stochastic matrices with sub-defect $k.$ Then $\omega_{n,k}$ partition $\omega_n,$  and $\omega_{n,0}=\Omega_n.$ Let $B\in \omega_{n,k},$ then denote $B^{comp}\in \Omega_{n+k}$ the minimal doubly stochastic completion of $B$ obtained by the method above.

Given $B=(b_{ij})\in \omega_{n,k}$ to obtain a minimal doubly stochastic completion of $B,$ one can append rows and columns to $B$ as the following:
\begin{enumerate}
\item Construct
$$\tilde{B}=\begin{pmatrix}b_{11} & b_{12} & \ldots & b_{1n} & x_1 \\ b_{21} & b_{22} & \ldots & b_{2n} & x_2 \\
\vdots & \vdots & \ddots & \vdots & \vdots  \\b_{n1} & b_{n2} & \ldots & b_{nn} & x_n  \\  y_1 & y_2 & \ldots & y_n & z\end{pmatrix}$$

where $$x_1=1-r_1(B),\ \  y_1=1-c_1(B)$$ and $$x_i=\min\{1-r_i(B),1-\sum_{j=1}^{i-1}x_j\}, \ \ y_i=\min\{1-c_i(B),1-\sum_{j=1}^{i-1}y_j\}$$ for $i=2,3,\ldots,n$ and $$z=\min\{1-\sum_{i=1}^n x_i, 1-\sum_{i=1}^n y_i\}.$$

\item If $B$ is doubly stochastic, then one doesn't want to do the construction. If $\tilde{B}$ is a doubly stochastic matrix, then it implies that $sd(B)=1$ and $\tilde{B}$ is a minimal completion. Otherwise, once show that $\tilde{B}\in \omega_{n+1,k+2}$ and repeat the previous step.
\end{enumerate}

This algorithm gives a minimal completion because each row and column appended except the last row and the last column has sum $1$ due to the way to choose $x_i$ and $y_i$ for $i=1,2,\ldots,n$ in step 1 and this is the largest step we can have to complete it.

Another important ingredient is a known formal for the dimension of a face of $\Omega_n.$ A matrix is called a  ${\bf (0,1)}$ matrix or a {\bf binary} matrix if all its entries are either $0$ or $1.$ An $n \times n$ binary matrix $A=[a_{ij}]$ has {\bf total support} if $a_{rs} = 1$ implies there exists an $n\times n$ permutation matrix $P = [p_{ij}]$ with $p_{rs} = 1$ and $P\leq A$ (\cite{sinkhorn1967}). Let $S=[s_{ij}]\in \omega_n,$ then the binary matrix $A$ corresponding to $S,$ denoted by $\mathcal{B}(S),$ is defined as $a_{ij}=1$ if $s_{ij}\neq 0.$ A fact is that for any $S\in \Omega_n,$ the corresponding binary matrix $\mathcal{B}(S)$ has total support due to Birkhoff's theorem. A doubly stochastic matrix $A$ is fully indecomposable, provided there do not exist permutation matrices $P$ and $Q$ such that $PAQ=A_1\oplus A_2$ (\cite{FENNER1971}).
\begin{theorem}  \cite{PM1965} A binary matrix $A$ has total support if and only if there exist permutation matrices $P$ and $Q$ such that $PAQ$ is a direct sum of fully indecomposable matrices.

\end{theorem}

Given an $n\times n$ binary matrix $A=[a_{ij}]$ which has total support, the {\bf face} of $\Omega_n$ corresponding to $A,$ denoted by $\mathcal{F}(A),$  is the set of all matrices $S=[s_{ij}]\in \Omega_n$ with $s_{ij}\leq a_{ij}$ for $i,j=1,2,\ldots, n.$

\begin{theorem}\label{BThm}(Corollary 2.6, \cite{BRU1977})  Let A be an $n \times n$  binary matrix with total support, let $P$ and $Q$ be permutation matrices such that $PAQ = A_1\oplus
\cdots\oplus A_t,$
where $A_i$ is a fully indecomposable  matrix for $i = 1,\ldots,t$. Then
$$\dim \mathcal{F}(A) = \sigma(A) - 2n+ t.$$ where $\dim \mathcal{F}(A)$ is the dimension of $\mathcal{F}(A)$.
\end{theorem}


\section{Main Results}

\begin{lemma}Let $D\in \omega_{n,k}$ and write $$D^{comp}=\begin{pmatrix}D  & X \\ Y & Z\end{pmatrix}$$ where $X,Y$ and $Z$ have dimension $n\times k, k\times n$ and $k\times k$ respectively.  Then \begin{enumerate}[label=(\roman*)]

\item $k-1 <\sigma(X)=\sigma(Y)\leq k;$
\item each column of $X$ contains at least one positive element and $c_i(X)=1$ for $i=1,\ldots,k-1;$
\item each row of $Y$ contains at least one positive element and $r_i(Y)=1$ for $i=1,\ldots,k-1.$
\end{enumerate}
\end{lemma}
\begin{proof} ($i$) Since $D\in \omega_{n,k},$ $n-k \leq  \sigma(D) < n-k+1.$ Since $D^{comp}\in \Omega_{n+k},$ $\sigma(D)+\sigma(X)=n$ because $D$ and $X$ partition the first $n$ rows of a doubly stochastic matrix $D^{comp}.$  Therefore
$$  k-1< \sigma(X)=n-\sigma(D)\leq k.$$ Similarly, $D$ and $Y$ partition the first $n$ columns of a doubly stochastic matrix $D^{comp},$ so $$k-1 <\sigma(X)=\sigma(Y)\leq k.$$

($ii$) Since $D\in \omega_{n,k},$ $n-k \leq \sigma(D)< n-k+1;$ Since $X$ is a submatrix of a doubly stochastic matrix $D^{comp},$ $c_i(X)\leq 1;$ Since $D$ and $X$ compose the first $n$ rows of a doubly stochastic matrix, $\sigma(D)+\sigma(X)=n$ and hence $\sigma(X)>k-1;$ Since $X$ has $k$ columns, $c_i(X)>0$ for all $i=1,2,\ldots,k.$ Suppose for some $1\leq t <k,$ $c_t(X)<1.$ Then all rows of $D$ have been completed, otherwise one should increase the values of some element in this column to complete those rows. Hence $c_i(X)=0$ for $i=t+1,t+2,\ldots, k$. Therefore $$\sigma(X)=\sum_{i=1}^t c_i(X)>k-1$$ which is impossible since there are only $t<k$ non-zero columns and all of them have sum at most $1.$

($iii$) Take the transpose of $D^{comp}$ to convert to ($ii$).
\end{proof}

\begin{lemma}\label{lmMain} Let $D\in \omega_{n,k}$ and write $$D^{comp}=\begin{pmatrix}D  & X \\ Y & Z\end{pmatrix}$$ where $X,Y$ and $Z$ have dimension $n\times k, k\times n$ and $k\times k$ respectively.  Then
\begin{enumerate}[label=(\roman*)]
\item $\sigma(\mathcal{B}(X))\leq n+k-1;$

\item $\sigma(\mathcal{B}(Y))\leq n+k-1;$

\item $\sigma(\mathcal{B}(Z))\leq 1.$
\end{enumerate}
and hence $\sigma(\mathcal{B}(D^{comp}))\leq \sigma(\mathcal{B}(D))+2(n+k)-1\leq n^2+2(n+k)-1.$
\end{lemma}

\begin{proof}  ($i$) Denote $$X=\begin{pmatrix}x_{11} & x_{12} & \ldots & x_{1k} \\  x_{21} & x_{22} & \ldots & x_{2k} \\ \vdots & \vdots & \ddots & \vdots \\ x_{n1} & x_{n2} & \ldots & x_{nk} \end{pmatrix}.$$ Due to the way that $D^{comp},$ $x_{11}=1-r_1(D)$ is constructed and $x_{1j}=0$ for $j=2,3,\ldots,k$ since $D^{comp}\in \Omega_{n+k}.$ So $X$ must has the form
$$X=\begin{pmatrix}1-r_1(D) & 0 & \ldots & 0 \\  x_{21} & x_{22} & \ldots & x_{2k} \\ \vdots & \vdots & \ddots & \vdots \\ x_{n1} & x_{n2} & \ldots & x_{nk} \end{pmatrix}.$$ Unless the first row of $D$ is zero for which the first column of $X$ is zero except $x_{11}=1,$ we will choose $x_{21}=\min\{1-r_2(D), 1-x_{11}\}\geq 0.$ If $x_{21}=1-r_2(D)<1-x_{11},$ then the second row is completed and $x_{2j}=0$ for $j=2,3,\ldots,k.$ If $x_{21}=1-x_{11}<1-r_2(D),$ then the first column of $X$ is completed, so $x_{j1}=0$ for $j=3,4,\ldots, n.$ If $x_{21}=1-r_2(D)<1-x_{11},$ then both second row and the first column of $X$ are completed.  Namely, $X$ must have the forms
$$X=\begin{pmatrix}1-r_1(D) & 0 & \ldots & 0 \\  1-r_2(D) & 0 & \ldots & 0 \\ x_{31} & x_{32} & \ldots & x_{3k} \\ \vdots & \vdots & \ddots & \vdots \\ x_{n1} & x_{n2} & \ldots & x_{nk} \end{pmatrix}$$
for which $x_{31}$ will be chosen next; or
$$X=\begin{pmatrix}1-r_1(D) & 0 & \ldots & 0 \\  1-(1-r_1(D)) & x_{22} & \ldots & x_{2k} \\ 0 & x_{32} & \ldots & x_{3k} \\ \vdots & \vdots & \ddots & \vdots \\ 0 & x_{n2} & \ldots & x_{nk} \end{pmatrix}$$
for which $x_{22}$ will be chosen next; or
$$X=\begin{pmatrix}1-r_1(D) & 0 & \ldots & 0 \\  1-r_2(D)=1-(1-r_1(D)) & 0 & \ldots & 0 \\ 0 & x_{32} & \ldots & x_{3k} \\ \vdots & \vdots & \ddots & \vdots \\ 0 & x_{n2} & \ldots & x_{nk} \end{pmatrix}$$ for which $x_{32}$ will be chosen next.

In general, once $x_{st}$ is determined, it may complete the row in which case  $x_{s+1,t}$ is determined next; or it may complete the column in which case $x_{s,t+1}$ is determined next; or it may complete both the row and column in which case $x_{s+1,t+1}$ is determined next.

Therefore $X$ must look like
$$\begin{pmatrix} *  &   & & \\ * & * && \\ & * && \\ & * && \\ & * & * && \\ && *& *& \\ &&& * & \\&&& * & \\ &&& * & \\&&& * &\\&&& * &*\\ &&& &* \end{pmatrix}$$
where $*$ are non-negative entries and other entries are zero. Each row contains at most two positive entries and at most $k-1$ row contains two positive entries which implies $$\sigma(\mathcal{B}(X))\leq n+k-1.$$

($ii$) Take the transpose of $D^{comp},$ then it is converted to ($i$).

($iii$) Due to the way  $D^{comp}$ is constructed, once $x_{st}$ is determined, one determine $x_{s,t+1}$ only when the column containing $x_{st}$ is completed. This means that $\sum_{i=1}^s x_{it}=1,$ so $x_{jt}=0$ for $j=s+1,\ldots,n.$ Note that the first $k-1$ columns in $X$ must be completed implying that all the first $k-1$ columns in $Z$ are zero. Similarly, the first $k-1$ rows are zero due to the structure of $Y.$ Hence the only entry in $Z$ which can be positive is $Z_{kk}=D^{comp}_{n+k,n+k}.$ Namely $\sigma(\mathcal{B}(Z))\leq 1.$

 Combine all of $(i), (ii)$ and $(iii)$, we have $\sigma(\mathcal{B}(D^{comp}))=\sigma(\mathcal{B}(D))+\sigma(\mathcal{B}(X))+\sigma(\mathcal{B}(Y))+\sigma(\mathcal{B}(Z))\leq
 \sigma(\mathcal{B}(D))+2(n+k)-1\leq n^2+2(n+k)-1 .$

\end{proof}
The following example illustrates the procedure of obtaining $D^{comp}$ for a given doubly substochastic matrix $D$.

\begin{example}
Let $D=\begin{pmatrix}0.1 & 0 & 0.2 & 0.1 \\ 0 & 0.2 & 0.1 & 0 \\ 0.2 & 0 & 0 & 0.1 \\ 0.1 & 0.2 & 0.3 &0.2 \end{pmatrix}.$ Since $\sigma(D)=1.8,$ $D\in\omega_{4,3}$ and hence $$D^{comp}=\begin{pmatrix}0.1 & 0 & 0.2 & 0.1 & x_{11} & x_{12} & x_{13}\\ 0 & 0.2 & 0.1 & 0 & x_{21} & x_{22} & x_{23}\\ 0.3 & 0 & 0 & 0.1 & x_{31} & x_{32} & x_{33} \\ 0.1 & 0.2 & 0.3 &0.2 & x_{41} & x_{42} & x_{43} \\ y_{11} & y_{12} & y_{13} & y_{14} & z_{11} & z_{12} & z_{13}  \\ y_{21} & y_{22} & y_{23} & y_{24} & z_{21} & z_{22} & z_{23} \\ y_{31} & y_{32} & y_{33} & y_{34} & z_{31} & z_{32} & z_{33}\end{pmatrix} $$
We start by letting $x_{11}=1-r_1(D)=0.6.$ Since $x_{11}$ completes the first row, all elements on the right to $x_{11}$ must be $0.$
$$D^{comp}=\begin{pmatrix}0.1 & 0 & 0.2 & 0.1 & 0.6 & 0 & 0 \\ 0 & 0.2 & 0.1 & 0 & x_{21} & x_{22} & x_{23}\\ 0.2 & 0 & 0 & 0.1 & x_{31} & x_{32} & x_{33} \\ 0.1 & 0.2 & 0.3 &0.2 & x_{41} & x_{42} & x_{43} \\ y_{11} & y_{12} & y_{13} & y_{14} & z_{11} & z_{12} & z_{13}  \\ y_{21} & y_{22} & y_{23} & y_{24} & z_{21} & z_{22} & z_{23} \\ y_{31} & y_{32} & y_{33} & y_{34} & z_{31} & z_{32} & z_{33}\end{pmatrix} $$
Next, let $x_{21}=\min\{1-r_2{D},1-x_{11}\}=\min \{  0.7 ,0.4\}=0.4$ which completes the fourth column of $D^{comp},$  so all elements below $x_{21}$ must be $0.$
$$D^{comp}=\begin{pmatrix}0.1 & 0 & 0.2 & 0.1 & 0.6 & 0 & 0 \\ 0 & 0.2 & 0.1 & 0 & 0.4 & x_{22} & x_{23}\\ 0.2 & 0 & 0 & 0.1 & 0 & x_{32} & x_{33} \\ 0.1 & 0.2 & 0.3 &0.2 & 0 & x_{42} & x_{43} \\ y_{11} & y_{12} & y_{13} & y_{14} & 0 & z_{12} & z_{13}  \\ y_{21} & y_{22} & y_{23} & y_{24} & 0 & z_{22} & z_{23} \\ y_{31} & y_{32} & y_{33} & y_{34} & 0 & z_{32} & z_{33}\end{pmatrix} $$
Next, let $x_{22}=1-r_2(D)-x_{21}=0.3$ which completes the second row and hence $x_{23}=0.$

$$D^{comp}=\begin{pmatrix}0.1 & 0 & 0.2 & 0.1 & 0.6 & 0 & 0 \\ 0 & 0.2 & 0.1 & 0 & 0.4 & 0.3 & 0 \\ 0.2 & 0 & 0 & 0.1 & 0 & x_{32} & x_{33} \\ 0.1 & 0.2 & 0.3 &0.2 & 0 & x_{42} & x_{43} \\ y_{11} & y_{12} & y_{13} & y_{14} & 0 & z_{12} & z_{13}  \\ y_{21} & y_{22} & y_{23} & y_{24} & 0 & z_{22} & z_{23} \\ y_{31} & y_{32} & y_{33} & y_{34} & 0 & z_{32} & z_{33}\end{pmatrix} $$
Next, let $x_{32}=\min\{1-r_3(D)-x_{31}, 1-x_{12}-x_{22}\}=0.7$ which completes both the third row and the fifth column, so all elements below and on the right to $x_{32}$  must be $0.$

  $$D^{comp}=\begin{pmatrix}0.1 & 0 & 0.2 & 0.1 & 0.6 & 0 & 0 \\ 0 & 0.2 & 0.1 & 0 & 0.4 & 0.3 & 0 \\ 0.2 & 0 & 0 & 0.1 & 0 & 0.7 & 0 \\ 0.1 & 0.2 & 0.3 &0.2 & 0 & 0 & x_{43} \\ y_{11} & y_{12} & y_{13} & y_{14} & 0 & 0 & z_{13}  \\ y_{21} & y_{22} & y_{23} & y_{24} & 0 & 0 & z_{23} \\ y_{31} & y_{32} & y_{33} & y_{34} & 0 & 0 & z_{33}\end{pmatrix} $$
Next, let $x_{43}=1-r_4(D)=0.2$ to complete the fourth row.

  $$D^{comp}=\begin{pmatrix}0.1 & 0 & 0.2 & 0.1 & 0.6 & 0 & 0 \\ 0 & 0.2 & 0.1 & 0 & 0.4 & 0.3 & 0 \\ 0.2 & 0 & 0 & 0.1 & 0 & 0.7 & 0 \\ 0.1 & 0.2 & 0.3 &0.2 & 0 & 0 & 0.2 \\ y_{11} & y_{12} & y_{13} & y_{14} & 0 & 0 & z_{13}  \\ y_{21} & y_{22} & y_{23} & y_{24} & 0 & 0 & z_{23} \\ y_{31} & y_{32} & y_{33} & y_{34} & 0 & 0 & z_{33}\end{pmatrix} $$

  Next, take the transpose of $D^{comp}$ to obtain $Y$ and $Z$ by the same procedure and then take the transpose again to obtain $D^{comp}$ completely.
$$D^{comp}=\begin{pmatrix}0.1 & 0 & 0.2 & 0.1 & 0.6 & 0 & 0 \\ 0 & 0.2 & 0.1 & 0 & 0.4 & 0.3 & 0 \\ 0.2 & 0 & 0 & 0.1 & 0 & 0.7 & 0 \\ 0.1 & 0.2 & 0.3 &0.2 & 0 & 0 & 0.2 \\ 0.6 & 0.4 & 0 & 0 & 0 & 0 & 0  \\ 0 & 0.2 & 0.4 & 0.4 & 0 & 0 & 0 \\ 0 & 0 & 0 & 0.2 & 0 & 0 & 0.8\end{pmatrix} $$

\end{example}
\begin{theorem} \label{main} Let $A\in \omega_{n,k},$ then $A$ can be written as a convex combination of no more than
$$\sigma(\mathcal{B}(A))+t$$ subpermutation matrices, where $t$ is the number of fully indecomposable components of $A^{comp}.$
\end{theorem}
\begin{proof} Since $A^{comp}\in \Omega_{n+k},$ $$\dim \mathcal{F}(\mathcal{B}(A^{comp}))=\sigma(\mathcal{B}(A^{comp}))-2(n+k)+t$$ due to Theorem \ref{BThm}. Apply Lemma \ref{lmMain}, we have
 $$\dim \mathcal{F}(\mathcal{B}(A^{comp}))=\sigma(\mathcal{B}(A^{comp}))-2(n+k)+t\leq \sigma(\mathcal{B}(A))+t-1,$$ meaning that $A^{comp}$ can be written as a convex combination of no more than  $$\sigma(\mathcal{B}(A))+t$$ permutation matrices. Therefore $A$ can be written as a convex combination of no more than $\sigma(\mathcal{B}(A))+t$ subpermutation matrices.
\end{proof}

The bound given by Theorem \ref{main} is actually tight. Here is an example in which the equality holds.
\begin{example} \label{ex1}

Let $A=\begin{pmatrix} \frac{7}{12} & 0 \\ \frac{1}{6} & \frac{1}{2} \end{pmatrix}\in \omega_{2,1}$ and its minimal completion is $$A^{comp}=\begin{pmatrix} \frac{7}{12} & 0& \frac{5}{12} \\ \frac{1}{6} & \frac{1}{2} & \frac{1}{3} \\ \frac{1}{4}& \frac{1}{2}&\frac{1}{4} \end{pmatrix} \in \Omega_3$$ and $$A^{comp}=\frac{1}{6}\cdot\begin{pmatrix}0 & 0 & 1 \\ 1 & 0 & 0 \\ 0 & 1 & 0 \end{pmatrix}+\frac{1}{4}\cdot\begin{pmatrix}1 & 0 & 0 \\ 0 & 1 & 0 \\ 0 & 0 & 1 \end{pmatrix}+\frac{1}{4}\cdot\begin{pmatrix}0 & 0 & 1 \\ 0 & 1 & 0 \\ 1 & 0 & 0 \end{pmatrix} +\frac{1}{3}\cdot\begin{pmatrix}1 & 0 & 0 \\ 0 & 0 & 1 \\ 0 & 1 & 0 \end{pmatrix}$$ and hence a convex combination of subpermutation matrices of $A$ is
\begin{equation}\label{convex}A=\frac{1}{6}\cdot\begin{pmatrix}0 & 0  \\ 1 & 0  \end{pmatrix}+\frac{1}{4}\cdot\begin{pmatrix}1 & 0  \\ 0 & 1 \end{pmatrix} +\frac{1}{4}\cdot\begin{pmatrix}0 & 0  \\ 0 & 1  \end{pmatrix} +\frac{1}{3}\cdot\begin{pmatrix}1 & 0  \\ 0 & 0 \end{pmatrix}.\end{equation}

In addition, suppose that \begin{eqnarray} \nonumber A&=&x_1\begin{pmatrix}1 & 0  \\ 0 & 0  \end{pmatrix}+x_2\cdot\begin{pmatrix}0 & 1  \\ 0 & 0 \end{pmatrix} +x_3\begin{pmatrix}0 & 0  \\ 1 & 0  \end{pmatrix} +x_4\cdot\begin{pmatrix}0 & 0  \\ 0 & 1 \end{pmatrix} \\
\nonumber &+&x_5\begin{pmatrix}1 & 0  \\ 0 & 1  \end{pmatrix}+x_6\cdot\begin{pmatrix}0 & 1  \\ 1 & 0 \end{pmatrix} +x_7\begin{pmatrix}0 & 0  \\ 0 & 0  \end{pmatrix}\\
\nonumber &=& \begin{pmatrix}x_1+x_5 & x_2+x_6 \\ x_3+x_6 & x_4+x_5  \end{pmatrix}.
\end{eqnarray}
With $\sum_{i=1}^7 x_i=1,$ we have a system corresponding to the augmented matrix
$$\begin{pmatrix}1 & 0 & 0 & 0 & 1 & 0 & 0 & \frac{7}{12} \\ 0 & 1 & 0 & 0 & 0 & 1 & 0 & 0 \\ 0 & 0 & 1 & 0 & 0 & 1 & 0 & \frac{1}{6} \\ 0 & 0 & 0 & 1 & 1 & 0 & 0 & \frac{1}{2}\\ 1 & 1 & 1 & 1 & 1 & 1 & 1 & 1\end{pmatrix}$$ which is equivalent to $$\begin{pmatrix}1 & 0 & 0 & 0 & 0 & -1 & 1 & \frac{5}{6} \\ 0 & 1 & 0 & 0 & 0 & 1 & 0 & 0 \\ 0 & 0 & 1 & 0 & 0 & 1 & 0 & \frac{1}{6} \\ 0 & 0 & 0 & 1 & 0 & -1 & 1 & \frac{3}{4}\\ 0 & 0 & 0 & 0 & 1& 1 & -1 & \frac{1}{4}\end{pmatrix}.$$
Each row has a pivot position, so $x_6$ and $x_7$ are free variables. Since $x_i\geq 0$ for $i=1,\ldots,7,$ the second row implies that $x_2=x_6=0.$ The solution to this system is that

  $$\left\{
\begin{aligned}
x_1&=\frac{5}{6}-x_7\\
x_2&=0\\
x_3&=\frac{1}{6}\\
x_4&=\frac{3}{4}-x_7\\
x_5&=\frac{1}{4}+x_7\\
x_6&=0\\
x_7&=x_7
\end{aligned}
\right.$$
It is clear that $x_3 \neq 0$ and one can choose $x_7$ to make one of $x_1,x_4,x_5$ and $x_7$ to be zero.Therefore at least four of these seven coefficients are non-zero. If one let $x_7=0,$ then this specific solution gives (\ref{convex}).

According to Theorem \ref{main}, $A$ contains $3$ non-zero elements and $A^{comp}$ is not decomposable meaning that $A^{comp}$ is $1$ fully indecomposable component, so $A$ can be written as a convex combination of no more than $4$ subpermutation matrices while at least $4$  subpermutation matrices are needed as we showed above, so the equality holds.
\end{example}

Although Theorem \ref{main} is proven via minimal doubly stochastic completion of doubly substochastic matrices, the upper bound of the number of subpermutation matrices needed does not depend on the sub-defect at all. Let $B=[b_{ij}]\in \omega_n$ and $b=\max_{i}\{r_i(B)\}.$ Note that $\frac{1}{b}B\in\omega_n.$ If $b<1,$ then the sub-defect of $\frac{1}{b}B$ is possibly much less than the sub-defect of $B.$ But $\frac{1}{b}B$ needs the same number of subpermutation matrices or at most one more, the zero matrix, as $B$. If there exists a $b_{ij}=1,$ then one can permute $B$ as a direct sum of an identity matrix and a smaller doubly substochastic matrix for which one can play the trick to the doubly substochastic matrix. So sub-defect does not affect the number of subpermutation matrices needed. Here are two examples.

\begin{example}
Let $A=\begin{pmatrix} \frac{7}{12} & 0 \\ \frac{1}{6} & \frac{1}{2} \end{pmatrix}$ be the same matrix in Example \ref{ex1} and let $$B=\frac{1}{3}A=\frac{1}{3}\cdot \begin{pmatrix} \frac{7}{12} & 0 \\ \frac{1}{6} & \frac{1}{2} \end{pmatrix}=\begin{pmatrix} \frac{7}{36} & 0 \\ \frac{1}{18} & \frac{1}{6} \end{pmatrix}.$$ Although $sd(A)=1\neq 2=sd(B),$ since $$A=\frac{1}{6}\cdot\begin{pmatrix}0 & 0  \\ 1 & 0  \end{pmatrix}+\frac{1}{4}\cdot\begin{pmatrix}1 & 0  \\ 0 & 1 \end{pmatrix} +\frac{1}{4}\cdot\begin{pmatrix}0 & 0  \\ 0 & 1  \end{pmatrix} +\frac{1}{3}\cdot\begin{pmatrix}1 & 0  \\ 0 & 0 \end{pmatrix},$$
$$B=\frac{1}{3}A=\frac{1}{18}\cdot\begin{pmatrix}0 & 0  \\ 1 & 0  \end{pmatrix}+\frac{1}{12}\cdot\begin{pmatrix}1 & 0  \\ 0 & 1 \end{pmatrix} +\frac{1}{12}\cdot\begin{pmatrix}0 & 0  \\ 0 & 1  \end{pmatrix} +\frac{1}{9}\cdot\begin{pmatrix}1 & 0  \\ 0 & 0 \end{pmatrix}+\frac{1}{3}\cdot\begin{pmatrix}0 & 0  \\ 0 & 0 \end{pmatrix}.$$
Therefore, if two doubly substochastic matrices are scalar multiple of each other, the difference between the numbers of subpermutation matrices needed for their convex expansion is at most $1.$

\end{example}

Let $B=[b_{ij}]\in \omega_n.$ If $\max\{b_{ij}\}_{i,j=1}^n=1,$ then there exist permutation matrices $P$ and $Q$ such that $PBQ=I_k\oplus \tilde{B}$ where $I_k$ is a $k\times k$ identity for some $k\leq n.$  Hence $\tilde{B}$ and $B$ need the same number of subpermutation matrices for their convex expansions.
\begin{example}
Let Let $A=\begin{pmatrix} \frac{7}{12} & 0 \\ \frac{1}{6} & \frac{1}{2} \end{pmatrix}$ be the same matrix in Example \ref{ex1} and let $$B=I_2\oplus A=\begin{pmatrix}1 & 0 & & \\ 0& 1 && \\ & & \frac{7}{12} & 0 \\  &&  \frac{1}{6} & \frac{1}{2}\end{pmatrix}.$$ Since
$$A=\frac{1}{6}\cdot\begin{pmatrix}0 & 0  \\ 1 & 0  \end{pmatrix}+\frac{1}{4}\cdot\begin{pmatrix}1 & 0  \\ 0 & 1 \end{pmatrix} +\frac{1}{4}\cdot\begin{pmatrix}0 & 0  \\ 0 & 1  \end{pmatrix} +\frac{1}{3}\cdot\begin{pmatrix}1 & 0  \\ 0 & 0 \end{pmatrix},$$
$$B=\frac{1}{6}\cdot\begin{pmatrix}1& 0 & & \\ 0 & 1 & & \\ &&0 & 0  \\&& 1 & 0  \end{pmatrix}+\frac{1}{4}\cdot\begin{pmatrix}1& 0 & & \\ 0 & 1 & & \\&&1 & 0  \\&& 0 & 1 \end{pmatrix} +\frac{1}{4}\cdot\begin{pmatrix}1& 0 & & \\ 0 & 1 & & \\&&0 & 0  \\ &&0 & 1  \end{pmatrix} +\frac{1}{3}\cdot\begin{pmatrix}1& 0 & & \\ 0 & 1 & & \\ & & 1 & 0  \\ & & 0 & 0 \end{pmatrix}.$$

\end{example}

\section{Acknowledgement}
The author would thank the anonymous reviewers for critically reading the manuscript and suggesting substantial improvements.

\end{document}